\documentclass[11pt]{amsart}
\usepackage{amsfonts,latexsym,amsthm,amssymb,graphicx,mathdots}
\usepackage[all]{xy}
\usepackage{hyperref}

\DeclareFontFamily{OT1}{rsfs}{}
\DeclareFontShape{OT1}{rsfs}{n}{it}{<-> rsfs10}{}
\DeclareMathAlphabet{\mathscr}{OT1}{rsfs}{n}{it}

\newtheorem{theorem}{Theorem}[section]
\newtheorem{lemma}[theorem]{Lemma}
\newtheorem{corol}[theorem]{Corollary}
\newtheorem{prop}[theorem]{Proposition}

{\theoremstyle{definition} \newtheorem{defin}[theorem]{Definition}}
{\theoremstyle{remark} \newtheorem{remark}[theorem]{Remark}
\newtheorem{example}[theorem]{Example}}

\numberwithin{equation}{section}

\newcommand{\Abb}{{\mathbb{A}}}
\newcommand{\Pbb}{{\mathbb{P}}}
\newcommand{\Zbb}{{\mathbb{Z}}}
\newcommand{\cE}{{\mathscr E}}
\newcommand{\cF}{{\mathscr F}}
\newcommand{\cL}{{\mathscr L}}
\newcommand{\cO}{{\mathscr O}}
\newcommand{\pp}{p}
\newcommand{\Til}[1]{{\widetilde{#1}}}
\newcommand{\csm}{{c_{\text{SM}}}}
\newcommand{\cf}{{c_{\text{F}}}}
\newcommand{\qede}{\hfill$\lrcorner$}
\DeclareMathOperator{\rk}{rk}
\DeclareMathOperator{\codim}{codim}

\title{
Tensored Segre classes
}
\author{Paolo Aluffi}
\address{
Mathematics Department, 
Florida State University,
Tallahassee FL 32306, U.S.A.
}
\email{aluffi@math.fsu.edu}

\begin{document}

\begin{abstract}
We study a class obtained from the Segre class $s(Z,Y)$ 
of an embedding of schemes by incorporating the datum of a line 
bundle on $Z$. This class satisfies basic properties analogous to
the ordinary Segre class, but leads to remarkably simple formulas in
standard intersection-theoretic situations such as excess or 
residual intersections. We prove a formula for the behavior of this
class under linear joins, and use this formula to prove that a `Segre
zeta function' associated with ideals generated by forms of the same 
degree is a rational function. 
\end{abstract}

\maketitle

%%%

\section{Introduction}\label{sec:intro}

Segre classes of subschemes are fundamental ingredients in
Fulton-MacPherson intersection theory: the very definition of
intersection product may be given as a component of a class obtained
by capping a Segre class by the Chern class of a
bundle determined by the data (\cite[Proposition
  6.1(a)]{85k:14004}). Segre classes also have applications in the
theory of characteristic classes of singular varieties: both the {\em
  Chern-Mather\/} and the {\em Chern-Schwartz-MacPherson\/} class of a
hypersurface of a nonsingular variety may be written in terms of Segre
classes determined by the singularity subscheme of the hypersurface
(\cite[Proposition~2.2]{MR2020555}). Precisely because they carry so
much information, Segre classes are as a rule very challenging to
compute, and their manipulation often leads to overly complex formulas. 
The main goal of this note is to study a variation on the definition of 
Segre class which produces a class with essentially the same amount 
of information, but enjoying features that may simplify its computation
and often leads to much simpler expressions. For example, standard 
applications to enumerative geometry may be streamlined by the use of this 
`tensored' class; and we will use this notion to give an efficient proof of the
rationality of a {\em Segre zeta function\/} of a homogeneous ideal in
a polynomial ring, subject to the condition that the generators of the
ideal all have the same degree. Concrete applications of the class to 
intersection-theoretic computations are given in~\cite{MR3383478}, 
where several of its properties are stated without proof. The proofs of 
those properties may be found (among others) in this note.

We work over an algebraically closed field $k$. Throughout this note,
$Y$ will denote an algebraic variety over $k$,
and $Z$ will be a closed subscheme of $Y$.
Segre classes of subschemes
are defined as Segre classes of related {\em cones.\/}
Recall the definition, from~\cite[Chapter~4]{85k:14004}: for a closed
subscheme $Z$ of $Y$, with normal cone $C_ZY$, the Segre class of $Z$
in $Y$ is the class
\begin{equation}\label{eq:scdef}
s(Z,Y):= q_*\left(\sum_{i\ge 0} c_1(\cO_\Pbb(1))^i \cap [\Pbb]\right)\in A_*Z
\end{equation}
where $\Pbb=\Pbb(C_Z(Y\times \Abb^1))$ is the projectivization of the
normal cone of $Z \cong Z\times \{0\}$ in~$Y\times \Abb^1$, $q: \Pbb \to
Z$ is the projection, and $\cO_\Pbb(1)$ is the tautological line bundle on
$\Pbb$.  (The extra $\Abb^1$ factor takes care of the possibility that
$Z$ may equal $Y$.)

One motivation for the introduction of the class studied here is the observation
that there are two ingredients to the definition recalled above: the projective 
cone $\Pbb$ and the tautological line bundle $\cO_\Pbb(1)$. The 
scheme $\Pbb$ does {\em not\/} determine the line bundle~$\cO_\Pbb(1)$: 
even if $C_Z(Y\times \Abb^1)$ is a vector bundle $\cE$, twisting $\cE$ by a
line bundle $\cL$ determines an isomorphic projective bundle:
$\Pbb(\cE\otimes \cL)\cong \Pbb(\cE)$, but modifies the tautological
line bundle by a corresponding twist. In this sense the notation 
$\cO_\Pbb(1)$ is ambiguous, as different realizations of the scheme 
$\Pbb$ affect $\cO_\Pbb(1)$ by a twist by a line bundle. Implementing
this additional degree of freedom leads to classes that share many of
the standard properties of Segre classes, but are in certain situations
better behaved and easier to use.

Thus, we consider the datum of a subscheme $Z\subseteq Y$ as above,
together with a line bundle $\cL$ over $Z$.

\begin{defin}\label{def:tensored}
The {\em $\cL$-tensored} Segre class of $Z$ in $Y$ is
\[
s(Z,Y)^{\cL} := s(Z,Y) \otimes_{Y\times \Abb^1} \cL\quad.
\]
\end{defin}

This notion is essentially a particular case of the {\em
  twisted Segre operator\/} defined and studied by Steven Kleiman and
Anders Thorup in their work on mixed Buchsbaum-Rim multiplicities,
\cite[\S4]{MR1393259}; see~\S\ref{sec:tso}. The $\otimes$ operation
used in Definition~\ref{def:tensored} was introduced 
in~\cite[Definition~2]{MR96d:14004}; `tensoring' the class by 
$\cL$ essentially reproduces the effect of tensoring the tautological
line bundle $\cO_\Pbb(1)$ by $\cL^\vee$.
In particular, 
the ordinary Segre class $s(Z,Y)$ agrees with the class tensored 
by the trivial bundle: $s(Z,Y)=s(Z,Y)^\cO$. 
Definition~\ref{def:tensored} has the advantage that if $s(Z,Y)$ is
known, computing the tensored class does not require an explicit 
realization of the normal
cone $\Pbb$. Also, good properties of the $\otimes$ notation 
significantly help in manipulations of tensored classes. For example,
\cite[Proposition~2]{MR96d:14004} implies that
\begin{equation}\label{eq:mul}
s(Z,Y)^{\cL_1\otimes \cL_2} = s(Z,Y)^{\cL_1}\otimes_{Y\times \Abb^1}\cL_2
\end{equation}
for all line bundles $\cL_1$, $\cL_2$ on $Z$.

We will prove several basic properties of tensored classes:
\begin{itemize}
\item[(i)] If $Z\subseteq Y$ is a regular embedding, with normal
  bundle $N_ZY$, then $s(Z,Y)^\cL = (c(\cL) c(N_ZY\otimes
  \cL))^{-1}\cap [Z]$.  In particular, if $Z=D\subseteq Y$ is a
  Cartier divisor, then $s(D,Y)^{\cO(-D)} = (1+D+D^2+ \cdots)\cap
  [D]$.
\item[(ii)] The tensored Segre class $s(Z,Y)^\cL$ is preserved by
  birational morphisms and by flat morphisms.
\item[(iii)] If $Y=V$ is a nonsingular variety, the class $c(TV\otimes
  \cL) \cap s(Z,V)^\cL$ is determined
  by $Z$ and $\cL$, and independent of $V$.
\item[(iv)] Residual intersection: Suppose $Z$ contains a Cartier
  divisor $D$ in $Y$, and $R$ is the residual scheme to $D$ in $Z$ (in
  the sense of \cite[\S9.2]{85k:14004}). Then (omitting evident push-forwards)
\[
s(Z,Y)^\cL=s(D,Y)^\cL+s(R,Y)^{\cO(D)\otimes \cL}\quad.
\]
\item[(v)] Suppose $Y\subseteq \Pbb^n$, and let $H$ be a general
  hyperplane. Then
\[
s(H\cap Z, H\cap Y)^\cL = H \cdot s(Z,Y)^\cL\quad.
\]
\end{itemize}
Several of these properties were stated without proof
in~\cite{MR3383478}, and used therein to streamline intersection-theoretic 
computations. They are analogues (and formal consequences)
of properties satisfied by ordinary Segre classes. Because of
these properties, tools normally used to compute Segre classes can be 
applied to compute tensored classes directly. For example, one may 
blow-up $Y$ along $Z$, use (i) to compute the tensored Segre class of 
the exceptional divisor, and (ii) (birational invariance) to obtain the
tensored Segre class of $Z$ in $Y$. Also note that, by (iv), the ordinary
residual formula for Segre classes takes the simple form
\begin{equation}\label{eq:resin}
s(Z,X)=s(D,X)+s(R,X)^{\cO(D)}
\end{equation}
and by (i) and (iv),
\[
s(Z,X)^{\cO(-D)} = (1+D+D^2+\cdots)\cap [D] + s(R,X)\quad.
\]
These examples illustrate the notational advantage of using tensored classes:
the reader is invited to compare~\eqref{eq:resin} with the standard formulation
of the residual formula in~\cite[Proposition~9.2]{85k:14004}.

Tensored Segre classes arise naturally in enumerative geometry. A
template situation in characteristic~$0$ may be as follows. Consider 
the linear system of
hypersurfaces of $\Pbb^n$ of degree $d$ containing a given scheme
$Z$. For $k=0,\dots, n$ we may ask for the number $N_k$ of points of
intersection of $k$ general such hypersurfaces and $n-k$ general
hyperplanes, {\em in the complement of $Z$.\/} (By Bertini's theorem, 
these intersection points will count with 
multiplicity~$1$.) We will prove:

\begin{theorem}\label{thm:enum}
With notation as above,
\begin{equation}\label{eqn:enum}
\iota_*\left(s(Z,\Pbb^n)^{\cO(-d)}\right) = \sum_{k=0}^n (d^k-N_k) [\Pbb^{n-k}]
\end{equation}
where $\iota: Z\hookrightarrow \Pbb^n$ is the inclusion.
\end{theorem}

For example, the problem of computing characteristic numbers of degree-$r$ 
plane curves fits this template: the projective space $\Pbb^n$, with $n=r(r+3)/2$, 
parametrizes degree-$r$
plane curves; the linear system is spanned by the hypersurfaces of degree 
$d=2r-2$ parametrizing curves that are tangent to lines; and $Z\subseteq \Pbb^n$ 
is a scheme supported on the set of non-reduced curves. In this case, the 
numbers $N_k$ are the `characteristic numbers' of the family of degree~$r$ plane 
curves. (They are known for $r\le 4$, \cite{MR1718648}; the problem of their 
computation is completely open for $r\ge 5$.) It is well known that the problem
is equivalent to the problem of computing the Segre class $s(Z,\Pbb^n)$; 
Theorem~\ref{thm:enum} makes this fact completely transparent.

By the same token, Theorem~\ref{thm:enum} can be used to compute 
Segre classes, in low dimension: if the hypersurfaces are general elements of the linear
system of hypersurfaces of degree~$d$ containing a given scheme $Z$, 
then a computer algebra system can be used to evaluate the numbers $N_k$, 
giving $\iota_*\left(s(Z,\Pbb^n)^{\cO(-d)}\right)$ by~\eqref{eqn:enum}, and
$\iota_* s(Z,\Pbb^n)$ may then be obtained by tensoring by $\cO(d)$, 
making use of~\eqref{eq:mul}. This strategy is reminiscent of the algorithm (via residual
schemes) introduced by Eklund, Jost, Peterson (\cite{EJP}).
An example will be given in~\S\ref{sec:equiv} (Example~\ref{ex:expco}).

We also note the following consequence of Theorem~\ref{thm:enum}.
\begin{corol}\label{cor:constraint}
Assume $Z\subseteq \Pbb^n$ may be defined by an ideal generated by polynomials of degree $d$.
Then $s(Z,\Pbb^n)^{\cO(-d)}$ is effective.
\end{corol}

\begin{example}\label{ex:effe}
Let $Z$ be the Veronese surface in $\Pbb^5$. Then $\iota_* s(Z,\Pbb^5)
= 4[\Pbb^2] -18[\Pbb^1] +51[\Pbb^0]$ is not effective, but as the Veronese
surface is cut out by quadrics,
\[
\iota_* s(Z,\Pbb^5)^{\cO(-2)} = 4[\Pbb^2] +14 [\Pbb^1] + 31 [\Pbb^0]
\]
is effective.
\qede\end{example} 
Ampleness considerations imply that the class $(1+dH)^{n+1} s(Z,\Pbb^n)$ 
is effective.
Corollary~\ref{cor:constraint} also implies this fact, as we note 
in~\S\ref{sec:equiv}; hence it is a stronger constraint.
Further constraints on the degrees of the components of $s(Z,\Pbb^n)^{\cO(-d)}$
may be derived from Theorem~\ref{thm:enum} by applying a theorem of
June Huh; see~Remark~\ref{rem:Huh}.

Theorem~\ref{thm:enum} will follow from a re-writing of
Fulton-MacPherson intersection product for the intersection of a
collection of linearly equivalent effective Cartier divisors
$X_1,\dots, X_m$ in a variety $V$.  Let $\cO(X)$ be the (common) line
bundle of these divisors, and assume that $Z$ is a union of connected
components of the intersection $X_1\cap\cdots \cap X_m$.

\begin{theorem}\label{thm:equiv}
The contribution of $Z$ to the intersection product $X_1\cdots X_m$
equals the component of dimension $\dim V-m$ in $s(Z,V)^{\cO(-X)}$.
\end{theorem}

This result is a formal consequence of known formulae
(cf.~\cite[Example 6.1]{85k:14004}) and of properties of the $\otimes$
operation from \cite[\S2]{MR96d:14004}. The reason we find
Theorem~\ref{thm:equiv} remarkable is that the class
$s(Z,V)^{\cO(-X)}$ {\em does not depend\/} on the number $m$ of
hypersurfaces: if $Z$ is a collection of components of the
intersection of more hypersurfaces from the same linear system, then
its contribution to the corresponding intersection product is simply
evaluated by terms of higher codimension in {\em the same\/} class
$s(Z,V)^{\cO(-X)}$. In fact, we can show (Theorem~\ref{thm:ls}) that
the contribution supported on subvarieties of $Z$ for the intersection 
product of {\em
  any\/} number of general elements of the linear system equals the
term of appropriate dimension in $s(Z,V)^{\cO(-X)}$. This fact is
responsible for the particularly simple form taken by
Theorem~\ref{thm:enum}.\smallskip

A similar notational advantage occurs in computing Segre classes of
joins with linear subspaces in projective space.  Let $Z$ be a
subscheme of $\Pbb^n$; we may assume $Z$ is defined by a (possibly
non-saturated) ideal generated by homogeneous polynomials $F_1,\dots,
F_m$ {\em of the same degree~$d$\/} in $k[x_0,\dots, x_n]$. For $N\ge
n$, we consider the subscheme $Z^{(d)}_N$ of $\Pbb^N$ defined by the
ideal generated by the $F_i$, viewed as polynomials in $k[x_0,\dots,
  x_N]$. Geometrically, $Z^{(d)}_N$ is the join of
$Z\subseteq\Pbb^n\subseteq \Pbb^N$ and a subspace $\Pbb^m$, $m=N-n-1$
complementary to $\Pbb^n$; but the scheme structure of $Z^{(d)}_N$
along the `vertex' $\Pbb^m$ depends on the choice of the degree~$d$.

\begin{example}\label{ex:lembp}
Let $Z$ be a point in $\Pbb^1$. Then $Z^{(1)}_2$ is a reduced line in
$\Pbb^2$, while $Z^{(2)}_2$ is a line with an embedded point.
\qede\end{example}

An analogous linear join operation may also be defined at the level of Chow
groups: as above, let $\Pbb^m \subseteq\Pbb^N$ be a complementary
subspace to $\Pbb^n$; and if $W\subseteq Z\subseteq \Pbb^n$ is a
subvariety, let $W\vee \Pbb^m$ denote the cone over $W$ with vertex
along $\Pbb^m$, a subvariety of~$Z^{(d)}_N$. This correspondence
passes through rational equivalence, hence it defines a map $\alpha
\mapsto \alpha \vee \Pbb^m$ from $A_*Z$ to~$A_* Z^{(d)}_N$.

\begin{theorem}\label{thm:cone}
In the situation detailed above (in particular, with $N=n+m+1$)
\begin{equation}\label{eq:coneform}
s(Z^{(d)}_N,\Pbb^N)^{\cO(-dH)} = \frac{d^{n+1}[\Pbb^m]}{1-dH} +
s(Z,\Pbb^n)^{\cO(-dH)} \vee \Pbb^m\quad.
\end{equation}
\end{theorem}

A `relative' version of Theorem~\ref{thm:cone}, proven here in 
Theorem~\ref{thm:conep}, is used in computations carried out
in~\cite{MR3383478}.
Another reason that motivates our interest in Theorem~\ref{thm:cone} 
is that this statement implies that the push-forward of the ordinary
Segre class to $\Pbb^N$ is of the form
\[
\iota_{N*} s(Z^{(d)}_N,\Pbb^N) = \frac{A(H)}{(1+dH)^{n+1}}\cap [\Pbb^N]
\]
where $A(H)$ is a certain polynomial {\em independent of $N$\/} and
with nonnegative coefficients (Theorem~\ref{thm:rf}). This is a particular 
case of the rationality of a `Segre zeta function' which may be associated 
with any homogeneous ideal $I$ of a polynomial ring; the case proven
here is the case in which all generators of $I$ have the same degree
(or, more generally, are elements of a linear system in a suitable relative
setting, cf.~Theorem~\ref{thm:conep}).
This zeta function appears to be quite interesting. For example, its poles 
record the degree of some, but in general not all, the elements of a minimal 
generating set of the ideal. The general case of this rationality statement 
will be discussed elsewhere.

The basic properties of tensored Segre classes are proven
in~\S\ref{sec:bp}.  Theorem~\ref{thm:equiv} is proven
in~\S\ref{sec:equiv}, together with its enumeratively-inspired
consequence Theorem~\ref{thm:enum}. The `relative' generalization of
Theorem~\ref{thm:cone} and the rationality of the Segre zeta 
function in the particular case considered here are 
discussed in~\S\ref{sec:cone}.\smallskip

As mentioned above, the class considered here may be viewed as an
application of the {\em twisted Segre operator\/} defined by
S.~Kleiman and A.~Thorup in~\cite{MR1393259}; see~\S\ref{sec:tso}. 
Properties (ii) and (iv) 
listed above are particular cases of more general statements for these
operators ((a), (b) in section (4.4), and Theorem~4.6 in \cite{MR1393259}, 
respectively). While `twisted Segre classes' may have been a natural choice
for the name of the classes considered here, we opted for `tensored'
for consistency with the terminology used in~\cite{MR3383478} and since
the term `twisted Segre class' is used in a different context in~\cite{MR2220087}.

Leendert van Gastel also considered classes 
defined similarly: one can interpret Corollary~3.6 in~\cite{MR1079843} 
as showing that the class of the {\em Vogel cycle\/} may be expressed 
as a tensored Segre class, up to multiplication by the Chern class of a 
line bundle. It would be interesting to compare this result with 
Theorem~\ref{thm:enum}. 

Finally, we note
that the Chern-Schwartz-MacPherson class of a hypersurface $X$ in a
nonsingular variety $M$ may be written as
\begin{equation}\label{eq:CSM}
\csm(X) = c(TM)\cap \left( s(X,M) + (s(JX,M)^{\cO(-X)})^\vee\right)
\end{equation}
where $JX$ denotes the {\em singularity subscheme\/} of $X$, and
$(\cdot)^\vee$ is the operation that changes the sign of the
components of the class $(\cdot)$ which have odd codimension in $M$.
This follows immediately from the definition and from
\cite[Theorem~I.4]{MR2001i:14009}.
\smallskip

{\em Acknowledgments.} This work was supported in part by the Simons foundation
and by NSA grants H98230-15-1-0027 and  H98230-16-1-0016.  The author is grateful 
to Caltech for hospitality while this work was carried out.

%%%

\section{Basic properties}\label{sec:bp}
As in the introduction, $Y$ denotes a variety over an algebraically
closed field $k$, embeddable in a nonsingular scheme.  (By
\cite[Lemma~4.2]{85k:14004}, the material extends without substantial
changes to the case in which $Y$ is a pure-dimensional $k$-scheme.)
In this section we prove the basic properties of tensored Segre
classes listed in the introduction.

We recall the definition of the tensored classes: for a closed
embedding $\iota: Z\hookrightarrow Y$, and for a line bundle $\cL$ on
$Z$, we let
\[
s(Z,Y)^\cL = s(Z,Y)\otimes_{Y\times \Abb^1} \cL\quad.
\]
Here we identify $Z$ with $Z\times \{0\}\subseteq Y\times \Abb^1$. The
$\otimes$ notation, borrowed from~\cite[\S2]{MR96d:14004}, acts on a
class $\alpha_k$ in $A_kZ$ by
\begin{equation}\label{eq:tende}
\alpha_k \otimes_{Y\times \Abb^1} \cL := c(\cL)^{-(\dim Y+1-k)} \cap \alpha_k\quad,
\end{equation}
and this definition is extended by linearity to the whole Chow group $A_*Z$.  Properties
of this operation are proven in~\cite[\S2]{MR96d:14004}. Recall that
$\alpha \mapsto \alpha\otimes_{Y\times \Abb^1} \cL$ defines an action
of Pic on $A_*Z$; in particular, $\alpha= (\alpha\otimes_{Y\times
  \Abb^1} \cL)\otimes_{Y\times \Abb^1} \cL^\vee$.  Hence, the ordinary
Segre class may be recovered from a tensored one by tensoring by the
dual line bundle:
\[
s(Z,Y) = s(Z,Y)^\cL \otimes_{Y\times \Abb^1} \cL^\vee\quad.
\]
We note that $s(Z,Y)^\cO=s(Z,Y)$. Also, $s(Z,Z)^\cL = c(\cL)^{-1}\cap [Z]$
if $Z$ is pure-dimensional.

\subsection{Twisted Segre operators}\label{sec:tso}
As mentioned in~\S\ref{sec:intro}, $s(Z,Y)^\cL$ may be expressed in
terms of the {\em twisted Segre operator\/} of~\cite{MR1393259}: with
notation as in~\cite[\S4.4]{MR1393259},
\[
s(Z,Y)^{\cL} = s(Z,\cL^\vee)[Y]\quad.
\]
Indeed, according to~\cite[(4.4.2)]{MR1393259} and with notation as
in~\S\ref{sec:intro},
\[
s(Z,\cL^\vee)[Y] = q_*\left(\sum_{i\ge 0} c_1(\cO_\Pbb(1)\otimes
q^*\cL^\vee)^i \cap [\Pbb]\right) =q_*\left( c(\cO_\Pbb(-1)\otimes
q^*\cL)^{-1}\cap [\Pbb]\right)
\]
(cf.~\eqref{eq:scdef}). By~\cite[Proposition~2]{MR96d:14004}, and
denoting by $\hat Y$ the blow-up of $Y\times \Abb^1$ along $Z\times
\{0\}$ (so that $[\Pbb]$ is a divisor in $\hat Y$),
\begin{align*}
q_*\left( c(\cO_\Pbb(-1)\otimes q^*\cL)^{-1}\cap [\Pbb]\right)
& =q_*\left( [\Pbb]\otimes_{\hat Y} (\cO_\Pbb(-1)\otimes q^*\cL)\right) \\
& =q_*\left( ([\Pbb]\otimes_{\hat Y}\cO_\Pbb(-1))\otimes_{\hat Y} q^*\cL\right) \\
& =q_*\left( ([\Pbb]\otimes_{\hat Y}\cO_\Pbb(-1))\right)\otimes_{Y\times \Abb^1} \cL \\
& =s(Z,Y)\otimes_{Y\times \Abb^1} \cL \\
& =s(Z,Y)^\cL\quad.
\end{align*}

\subsection{Regular embeddings}\label{sec:re}
\begin{itemize}
\item[(i)] If $Z\subseteq Y$ is a regular embedding, with normal
  bundle $N_ZY$, then
\begin{equation}\label{eq:regem}
s(Z,Y)^\cL = (c(\cL) c(N_ZY\otimes \cL))^{-1}\cap [Z]\quad.
\end{equation}
\end{itemize}

\begin{proof}
If $Z\subseteq Y$ is a regular embedding, then
$s(Z,Y)=c(N_ZY)^{-1}\cap [Z]$ by \cite[Proposition
  4.1(a)]{85k:14004}. Also note that in this case $Z$ is
pure-dimensional (as $Y$ is pure-dimensional by assumption).
Applying~\cite[Proposition~1]{MR96d:14004}, we obtain
\begin{align*}
s(Z,Y)^\cL & = s(Z,Y) \otimes_{Y\times \Abb^1} \cL 
= (c(N_ZY)^{-1}\cap [Z])\otimes_{Y\times \Abb^1} \cL \\
& = c(\cL)^{\codim_Z Y} c(N_ZY\otimes \cL)^{-1} \cap ([Z]\otimes_{Y\times \Abb^1} \cL) \\
& = c(\cL)^{\dim Y-\dim Z} c(N_ZY\otimes \cL)^{-1} \cap (c(\cL)^{-(\dim Y+1-\dim Z)}\cap [Z])
\end{align*}
with the stated result.
\end{proof}

\begin{example}
Let $Z$ be the complete intersection of $r$ linearly equivalent
effective Cartier divisors $X_1,\dots,X_r$; let $D$ be the common
divisor class of the hypersurfaces $X_i$. Then $N_ZY=\cO(D)^{\oplus
  r}$, hence $N_ZY\otimes \cO(-D)\cong \cO^{\oplus r}$, and
\eqref{eq:regem}~gives
\[
s(Z,Y)^{\cO(-D)} = c(\cO(-D))^{-1} \cap [Z] = (1 + D + D^2 + \cdots )\cap [Z]
=s(Z,Z)^{\cO(-D)}\quad,
\]
independently of $r$.
In particular, if $Z=D$ is a Cartier divisor, then $s(D,Y)^{\cO(-D)} =
s(D,D)^{\cO(-D)}=D+D^2 + D^3 + \cdots$ as stated in~\S\ref{sec:intro}.
\qede\end{example}

\subsection{Behavior under morphisms}\label{sec:bi}
\begin{itemize}
\item[(ii)] Let $\pi: Y' \to Y$ be a morphism of varieties, let 
$\rho:\pi^{-1}(Z)\to Z$ be the induced morphism, and $\Til\cL=\rho^*\cL$.
Then
\begin{itemize}
\item If $\pi$ is proper and onto, then $\rho_*(s(\pi^{-1}(Z),Y')^{\Til\cL})=\deg(Y'/Y) s(Z,Y)^\cL$.
\item If $\pi$ is flat, then $\rho^*(s(Z,Y)^\cL) = s(Z',Y')^{\Til\cL}$.
\end{itemize}
\end{itemize}

\begin{proof}
Both statements follows immediately from the analogous properties of ordinary
Segre classes, proven in\cite[Proposition~4.2(a)]{85k:14004}, and 
from the projection formula, which implies that
\[
\rho_* (\alpha \otimes_{\Til Y\times \Abb^1} \rho^*\cL) 
= \rho_*(\alpha)\otimes_{Y\times \Abb^1} \cL
\]
as is immediate from the definition of the $\otimes$ operation.
\end{proof}

For the corresponding (more general) facts for twisted Segre operators,
see~\cite[(4.4)]{MR1393259}.

Note that as a consequence of the first formula, tensored Segre classes are
invariant under birational maps, in the sense that 
$\rho_*\left(s(\pi^{-1}(Z),\Til Y)^{\Til\cL}\right)=s(Z,V)^\cL$
if $\pi$ is a proper birational morphism.

\subsection{Independence of a nonsingular ambient variety}
\begin{itemize}
\item[(iii)] If $Y=V$ is a nonsingular variety, the class $c(TV|_Z\otimes \cL) 
\cap s(Z,V)^\cL$ is independent of $V$; it is determined by $Z$ and $\cL$.
\end{itemize}

In fact, we will show that
\begin{equation}\label{eq:ful}
c(TV|_Z\otimes \cL) \cap s(Z,V)^\cL = c(\cL)^{\dim V}\cap (\cf(Z)\otimes_{V\times \Abb^1} \cL)
\end{equation}
where $\cf(Z)$ is the class defined
in~\cite[Example~4.2.6]{85k:14004}. This class only depends on $Z$,
and if $\alpha_k$ is a class of dimension~$k$, then
\[
c(\cL)^{\dim V}\cap (\alpha_k \otimes_{V\times \Abb^1} \cL) 
=c(\cL)^{\dim V}\cap (c(\cL)^{-(\dim V+1-k)}\cap \alpha_k)
=c(\cL)^{k-1}\cap \alpha_k
\]
is independent of $V$, so indeed~\eqref{eq:ful} verifies~(iii).

\begin{proof}[Proof of~\eqref{eq:ful}]
By \cite[Proposition~1]{MR96d:14004},
\begin{align*}
c(\cL)^{\dim V}& \cap (\cf(Z)\otimes_{V\times \Abb^1} \cL) = 
c(\cL)^{\dim V}\cap ((c(TV)\cap s(Z,V))\otimes_{V\times \Abb^1} \cL) \\
&= c(\cL)^{\dim V} \cap (c(\cL)^{-\dim V} c(TV\otimes \cL)\cap (s(Z,V)\otimes_{V\times \Abb^1} \cL)) \\
&= c(TV\otimes \cL)\cap s(Z,V)^\cL
\end{align*}
as needed.
\end{proof}

\subsection{Residual intersection}\label{sec:ri}
\begin{itemize}
\item[(iv)] Suppose $Z$ contains a Cartier divisor $D$ in $Y$, and let
  $R$ be the residual scheme to~$D$ in $Z$. Then
\begin{equation}\label{eq:ri}
s(Z,Y)^\cL=s(D,Y)^\cL+s(R,Y)^{\cO(D)\otimes \cL}\quad.
\end{equation}
\end{itemize}

\begin{proof}
This follows from the usual residual intersection formula, i.e.,
\cite[Proposition~9.2]{85k:14004}, in the formulation given
in~\cite[Proposition~3]{MR96d:14004}:
\[
s(Z,Y)=s(D,Y)+c(\cO(D))^{-1}\cap (s(R,Y)\otimes_Y \cO(D))\quad.
\]
This gives
\begin{align*}
s(Z,Y)^\cL &= s(Z,Y)\otimes_{Y\times \Abb^1} \cL\\
&=s(D,Y)\otimes_{Y\times \Abb^1} \cL+c(\cL)^{-1} \frac{c(\cL)}{c(\cO(D)\otimes \cL)}
\cap ((s(R,Y)\otimes_Y \cO(D))\otimes_Y \cL) \\
&\overset{*}=s(D,Y)^\cL + c(\cO(D)\otimes \cL)^{-1} (s(R,Y)\otimes_Y (\cO(D)\otimes \cL)) \\
&=s(D,Y)^\cL + s(R,Y)^{c(\cO(D)\otimes \cL)}
\end{align*}
as stated. Equality $*$ follows from~\cite[Proposition~2]{MR96d:14004}.
\end{proof}

The `additivity' formula~\eqref{eq:ri} will be used in the proof of 
Theorem~\ref{thm:conep}. It may also be obtained as a particular case
of additivity for twisted Segre operators, \cite[Theorem~4.6, (4.7.1)]{MR1393259}.

\subsection{General hyperplane sections}\label{ss:Ghs}
\begin{itemize}
\item[(v)] Suppose $Y\subseteq \Pbb^n$, and let $H$ be a general
  hyperplane. Then
\[
s(Z\cap H, Y\cap H)^\cL = H \cdot s(Z,Y)^\cL\quad.
\]
\end{itemize}

\begin{proof}
More generally, we can prove that if $D$ is a Cartier divisor of $Y$
intersecting properly every component of the normal cone of $Z$ in
$Y$, then
\[
s(D\cap Z, D)^\cL = D\cdot s(Z,Y)^\cL\quad.
\]
Indeed, it is easy to see that this is the case for ordinary Segre
classes (\cite[Lemma~4.1]{MR3415650}), so we only need to verify that
if $\alpha\in A_*Z$ and $D$ is a divisor of $Y$, then for all line
bundles~$\cL$
\[
(D\cdot \alpha)\otimes_{D\times \Abb^1} \cL = D\cdot
(\alpha\otimes_{Y\times \Abb^1} \cL)\quad.
\]
This is immediately checked for a pure-dimensional class by
using~\eqref{eq:tende}, hence it holds for all classes by linearity.
\end{proof}

%%%

\section{Intersection product}\label{sec:equiv}
In this section we use tensored Segre classes to give a reformulation
of the Fulton-MacPherson intersection product. This will again be a
formal consequence of the usual formulation of the product, but in
some situations the use of tensored classes yields particularly
simple expressions, cf.~Theorems~\ref{thm:enum} and~\ref{thm:equiv}.

The reformulation relies on the following observation concerning the
$\otimes$ operation used to defined tensored classes.

\begin{lemma}\label{lem:inde}
Let $A$ be a Chow class in a subscheme $Z$ of a pure-dimensional
scheme~$X$, and let $\cL$ be a line bundle on $Z$. Then the term of
dimension $\dim X-c$ in
\[
c(\cL)^{c-1}\cap (A\otimes_X \cL)
\]
equals the term of dimension $\dim X-c$ in $A$ (in particular, it is
independent of $\cL$).
\end{lemma}

\begin{proof}
Letting $A^{(i)}$ denote the part of $A$ of dimension $\dim X - i$,
\begin{align*}
c(\cL)^{c-1}\cap & (A\otimes_X \cL) = c(\cL)^{c-1} \left(\frac{A^{(0)}}{c(\cL)^0} 
+ \frac{A^{(1)}}{c(\cL)^1} + \frac{A^{(2)}}{c(\cL)^2} + \cdots\right) \\
&= c(\cL)^{c-1}\cap A^{(0)} + c(\cL)^{c-2}\cap A^{(1)} + \cdots
+ A^{(c-1)} + c(\cL)^{-1}\cap A^{(c)} + \cdots
\end{align*}
It is clear that the term of dimension $\dim X - c$ in this expression
is $A^{(c)}$, independently of~$\cL$.
\end{proof}

\begin{remark}
If $A=\frac{c(\cE)}{c(\cF)}\cap [X]$, with $\cE$, $\cF$ vector bundles of ranks
$e$, $f$ respectively, then Lemma~\ref{lem:inde} asserts that the term of 
codimension $c=e-f+1$ in $A$ does not change if we tensor both $\cE$ and
$\cF$ by a line bundle $\cL$. Indeed, using~\cite[Proposition~1]{MR96d:14004}
\begin{multline*}
c(\cL)^{c-1}\cap (A\otimes_X \cL) 
= c(\cL)^{e-f}\cap \left(\left(\frac{c(\cE)}{c(\cF)}\cap [X]\right)\otimes \cL\right)\\
= c(\cL)^{e-f}\cap \left(\frac{c(\cE\otimes \cL)}{c(\cL)^{e-f}c(\cF\otimes \cL)} \cap [X]\right)
=\frac{c(\cE\otimes \cL)}{c(\cF\otimes \cL)}\cap [X]\quad.
\end{multline*}
This recovers the result of~\cite{MR1348792}.
\qede\end{remark}

Now we consider a standard intersection template. Let $V$ be a
variety, $B\subseteq V$ a closed subscheme, and assume that the
inclusion $B\hookrightarrow Y$ is a regular embedding. Let $f: Y\to V$
be a morphism, and assume $Z\subseteq Y$ is a collection of connected 
components of
$f^{-1}(B)$. Let $g: Z\to B$ be the induced morphism.
\[
\xymatrix{
Z \ar[r] \ar[d]_g & Y \ar[d]^f \\
B \ar[r]  & V
}
\]

\begin{prop}\label{prop:fminv}
For all line bundles $\cL$ on $Z$, the contribution $(B\cdot Y)_Z$ of $Z$ 
to the Fulton-MacPherson intersection product $B\cdot Y$ is given by
\begin{equation}\label{eq:fminv}
(B\cdot Y)_Z=\{ c(g^* N_BV\otimes \cL) \cap s(Z,Y)^\cL \}_d
\end{equation}
where $\{\cdot \}_d$ denotes the term of dimension $d$, and $d=\dim Y-
\codim_B V$.
\end{prop}

The point of this statement is that the contribution of $Z$ to $B\cdot Y$,
and hence the right-hand-side of~\eqref{eq:fminv}, is
{\em independent of $\cL$;\/} thus, we may have the flexibility of
choosing a specific line bundle to simplify this
expression. Theorem~\ref{thm:equiv} will precisely be obtained in this
fashion.

\begin{proof}
By~\cite[\S6.1]{85k:14004}, 
$(B\cdot Y)_Z=\{ c(g^*N_BV)\cap s(Z,Y)\}_d$. Applying Lemma~\ref{lem:inde} to 
$A=c(g^*N_BV)\cap s(Z,Y)$, $X=Y\times \Abb^1$, and $c=\dim Y+1-d$, we obtain
\begin{align*}
(B\cdot Y)_Z &= \{c(g^* N_B V)\cap s(Z,Y)\}_d \\
&= \{ c(\cL)^{\dim Y-d}\cap ((c(g^*N_B V)\cap s(Z,Y))\otimes_{Y\times \Abb^1} \cL) \}_d \\
&\overset{*}= \{ c(\cL)^{\dim Y-d}c(\cL)^{-\codim_B V}c(g^*N_BV\otimes \cL)\cap 
(s(Z,Y)\otimes_{Y\times \Abb^1} \cL) \}_d \\
&= \{ c(g^*N_BV\otimes \cL)\cap s(Z,Y)^\cL\}_d
\end{align*}
as stated. Equality $*$ holds by \cite[Proposition~1]{MR96d:14004}.
\end{proof}

Next, we verify that Theorem~\ref{thm:equiv} follows from
Proposition~\ref{prop:fminv}.  Let $X_1,\dots, X_m$ be effective
Cartier divisors in a variety $V$; assume $\cO(X_i)$ is independent of
$i$, and let $\cO(X)$ denote this line bundle. We define the
intersection product $X_1\cdots X_m$ by applying the Fulton-MacPherson
definition (\cite[\S6.1]{85k:14004}) to the following Cartesian
diagram:
\[
\xymatrix{
X_1\cap \cdots \cap X_m \ar[r] \ar[d]_i & V \ar[d]^\Delta \\
X_1\times \cdots \times X_m \ar[r] & V\times \cdots \times V
}
\]
where the vertical map $\Delta$ is the diagonal embedding.

Let $Z$ be a union of connected components of the intersection 
$X_1\cap \cdots \cap X_m$. Theorem~\ref{thm:equiv} states that 
the contribution $(X_1\cdots X_m)_Z$ of $Z$ to the intersection 
product $X_1\cdots X_m$ is given by
\[
(X_1\cdots X_m)_Z=\{s(Z,V)^{\cO(-X)}\}_{\dim V-m}\quad.
\]

\begin{proof}[Proof of Theorem~\ref{thm:equiv}]
We have $Z\subseteq X_1\cap \cdots \cap X_m \overset i\hookrightarrow
X_1\times \cdots \times X_m$. Denote by $g$ this inclusion. We have
\[
g^* N_{X_1\times\cdots \times X_m} (V\times \cdots \times V) 
=\bigoplus_j N_{X_j}V|_Z \cong \cO(X)^{\oplus m}|_Z\quad.
\]
It follows that
\[
(g^* N_{X_1\times\cdots \times X_m} (V\times \cdots \times V))\otimes \cO(-X)|_Z \cong
\cO_Z^{\oplus m}\quad,
\]
and hence
\[
c(g^* N_{X_1\times\cdots \times X_m} (V\times \cdots \times V)) = 1\quad.
\]
The statement then follows immediately from Proposition~\ref{prop:fminv}.
\end{proof}

Next, we assume that $Z$ is cut out by a linear system $L\subseteq
H^0(Y,\cL)$ and that $Y$ is projective; let $H$ be the hyperplane
class on $Y$. For a class $\alpha\in A_kY$, we will let $\deg \alpha$
denote $\int_Y H^k\cdot \alpha$, i.e., the degree of the push-forward of
$\alpha$ to projective space. If $X_1, X_2,\dots$ are general elements
of $L$, we are interested in the contribution to $\deg (X_1\cdots
X_c)$ supported on (subvarieties of) $Z$, for all $c\ge 0$.  As mentioned
in~\S\ref{sec:intro}, this situation is motivated by enumerative
geometry: typically, the non-complete variety $Y\smallsetminus Z$ may
parametrize some type of geometric object, and the degree of the part
of $X_1\cdots X_c$ supported on $Y \smallsetminus Z$ will have
enumerative significance. This degree can be obtained by taking the
contribution supported on $Z$ away from the total degree of $X_1\cdots
X_c$; thus, this operation may be viewed as `performing intersection
theory in the non-complete variety $Y\smallsetminus Z$'.

\begin{theorem}\label{thm:ls}
For all $c\ge 0$, the contribution to $\deg (X_1\cdots X_c)$ supported
on $Z$ equals $\deg \{s(Z,Y)^{\cL^\vee}\}_{\dim Y - c}$.
\end{theorem}

\begin{proof}
Let $H_1,\dots, H_{\dim Y-c}$ be general hyperplanes, and let
\[
Z^{(c)} = H_1\cap \cdots \cap H_{\dim Y-c}\cap Z\quad, \quad
Y^{(c)} = H_1\cap \cdots \cap H_{\dim Y-c}\cap Y\quad;
\]
note that $\dim Y^{(c)} = c$.  The intersections $X_i^{(c)}:=X_i\cap
Y^{(c)}$ are general representative of the restriction of the linear
system $L$ to $Y^{(c)}$; this system cuts out $Z^{(c)}$. We have
\[
\deg (X_1\cdots X_c) = \int_Y (X_1^{(c)}\cdots X_c^{(c)})\quad.
\]
By Theorem~\ref{thm:equiv}, the contribution of $Z^{(c)}$ to
$X_1^{(c)}\cdots X_c^{(c)}$ is given by
\[
\int_Y \{s(Z^{(c)}_-,Y^{(c)})^{\cL^\vee}\}_0\quad,
\]
where $Z^{(c)}_-$ is the part of $X_1^{(c)}\cap\cdots\cap X_c^{(c)}$ 
supported within $Z^{(c)}$.  By \cite[Theorem~1.1(b)]{howmany},
$s(Z^{(c)}_-,Y^{(c)})=s(Z^{(c)},Y^{(c)})$, and hence
$s(Z^{(c)}_-,Y^{(c)})^{\cL^\vee}=s(Z^{(c)},Y^{(c)})^{\cL^\vee}$.
Further, by property (v) of tensored Segre classes (cf.~\S\ref{ss:Ghs}),
\[
s(Z^{(c)},Y^{(c)})^{\cL^\vee} = H^{\dim Y-c} \cdot s(Z,Y)^{\cL^\vee}\quad.
\]
It follows that the contribution to $X_1\cdots X_c$ supported on $Z$
has degree
\[
\int_Y H^{\dim Y-c} \cdot s(Z,Y)^{\cL^\vee}  = \deg \{s(Z,Y)^{\cL^\vee}\}_{\dim Y-c}
\]
as stated.
\end{proof}

Theorem~\ref{thm:enum} is a special case of Theorem~\ref{thm:ls},
where $Y=\Pbb^n$ and $\cL=\cO(d)$.

\begin{example}
For the problem of characteristic numbers of plane conics we have
$Y=\Pbb^5$, $d=2$; and $Z$ consists of the Veronese surface in
$\Pbb^5$ with its reduced structure.  Denoting by $h$ the hyperplane
in $Z\cong \Pbb^2$, the pull-back of $H$ to $Z$ equals $2h$, and we
have
\[
c(N_Z\Pbb^5) = \frac{(1+2h)^6}{(1+h)^3}\quad.
\]
By property (i) of tensored classes (cf.~\S\ref{sec:re}) we have that
\begin{align*}
s(Z,\Pbb^5)^{\cO(-2H)} &= (c(\cO(-4h)) c(N_Z\Pbb^5\otimes \cO(-4h)))^{-1}\cap [Z]
=\frac{(1+h-4h)^3}{(1-4h)(1+2h-4h)^6}\cap [Z]\\
&= \frac{(1-3h)^3}{(1-4h)(1-2h)^6}\cap [Z] =(1+7h+31h^2)\cap [Z]
\end{align*}
and it follows that
\[
\iota_* s(Z,\Pbb^5)^{\cO(-2H)} = 4[\Pbb^2]+14[\Pbb^1]+31[\Pbb^0]\quad.
\]
Using Theorem~\ref{thm:enum}, this says that the characteristic
numbers $N_k$ for smooth plane conics, that is, the number of conics
tangent to $k$ general lines and containing $5-k$ general points, must
be $1,2,2^2, 2^3-4,2^4-14,2^5-31= 1,2,4,4,2,1$ for $k=0,\dots, 5$.
\qede\end{example}

\begin{example}\label{ex:expco}
As mentioned in~\S\ref{sec:intro}, results such as
Theorems~\ref{thm:ls} or~\ref{thm:enum} may be used to compute Segre
classes. For example, consider the monomial scheme $Z$ defined by the
ideal $I=(x_1^2 x_2^6, x_1^3 x_2^4, x_1^4 x_2^3, x_1^5 x_2, x_1^7)$ in
$\Pbb^3$. Let $f_1, f_2, f_3$ be general degree-$8$ polynomials in
$I$, and let
\[
J_1=(f_1):I^\infty\quad, \quad J_2=(f_1,f_2):I^\infty\quad, \quad J_3=(f_1,f_2,f_3):I^\infty\quad.
\]
Macaulay2 can compute these ideals for `random' polynomials $f_i$, and
the degrees of the residual schemes $R_1$, $R_2$, $R_3$ defined by the 
ideals $J_1$, $J_2$, $J_3$:
\begin{equation}\label{eq:res8}
\deg R_1 = 6\quad, \quad \deg R_2 = 14\quad,\quad \deg R_3 = 30\quad.
\end{equation}
According to Theorem~\ref{thm:enum},
\[
\iota_* s(Z,\Pbb^3)^{\cO(-8)} = (8-6)[\Pbb^2]+(8^2-14)[\Pbb^1]+(8^3-30)[\Pbb^0]
=2[\Pbb^2]+50[\Pbb^1]+482[\Pbb^0]\quad.
\]
Letting $H$ denote the hyperplane class, it follows that
\[
\iota_* s(Z,\Pbb^3) = \frac{2[\Pbb^2]}{(1+8H)^2} + \frac{50[\Pbb^1]}{(1+8H)^3} + 
\frac{482[\Pbb^0]}{(1+8H)^4} = 2[\Pbb^2] + 18[\Pbb^1] - 334[\Pbb^0]\quad.
\]
Cf.~\cite[Example 1.2]{Scaiop} for a different computation of the same
class.
We note that by~\eqref{eq:mul}
\begin{align*}
s(Z,\Pbb^3)^{\cO(-d)} &= s(Z,\Pbb^3)^{\cO(-8)}\otimes_{\Pbb^3\times \Abb^1} \cO(8-d) \\
&=\frac{2[\Pbb^2]}{(1+(8-d)H)^2} + \frac{50[\Pbb^1]}{(1+(8-d)H)^3} + 
\frac{482[\Pbb^0]}{(1+(8-d)H)^4}\\
&=2[\Pbb^2]+(4d+18)[\Pbb^1]+(6d^2+54d-334)[\Pbb^0]\quad,
\end{align*}
and by Theorem~\ref{thm:enum} we can deduce that the degrees of the corresponding
residual schemes for degree-$d$ general polynomials in $I$ must be
\[
\deg R'_1 = d-2\quad, \quad \deg R'_2 = d^2-4d-18\quad,\quad \deg R'_3 = 
d^3-6d^2-54d+334\quad.
\]
Macaulay2 can confirm low degree specializations of this formula (e.g., $d=9$) in
this example. Therefore, the residual degrees for any one $d$ ($d=8$ in this case) 
determine the residual degrees for every~$d$.
\qede\end{example}

Corollary~\ref{cor:constraint} follows from Theorem~\ref{thm:enum}
and~\cite[Theorem 12.3]{85k:14004}, which ensures that all
contributions to an intersection product in projective space are
nonnegative: with notation as in~\S\ref{sec:intro}, $d^i-N_i\ge 0$ for
all $i$, therefore the class
\[
\iota_* s(Z,\Pbb^n)^{\cO(-d)} = (d^0-N_0) [\Pbb^n] + (d^1-N_1) [\Pbb^{n-1}] + \cdots + (d^n-N_n)
[\Pbb^0]
\]
is effective. We also note that it follows that
\begin{align*}
(1+dH&)^{n+1}\cap \iota_* s(Z,\Pbb^n) = (1+dH)^{n+1} \left(\iota_* s(Z,\Pbb^n)^{\cO(-dH)}
\otimes_{\Pbb^n\times\Abb^1} \cO(dH)\right) \\
&= (1+dH)^{n+1} \left(\frac{(d^0-N_0) [\Pbb^n]}{1+dH} + \frac{(d^1-N_1) [\Pbb^{n-1}]}{(1+dH)^2} 
+ \cdots + \frac{(d^n-N_n)[\Pbb^0]}{(1+dH)^{n+1}}\right) \\
&= (d^0-N_0)(1+dH)^n[\Pbb^n]+ (d^1-N_1)(1+dH)^{n-1} [\Pbb^{n-1}] + \cdots + (d^n-N_n)[\Pbb^0]
\end{align*}
is also necessarily an effective class.

\begin{remark}\label{rem:Huh}
Further constraints on the degrees of the components of $s(Z,\Pbb^n)^{\cO(-d)}$
follow from Theorem~\ref{thm:enum} and a theorem of June Huh. Specifically, 
assume that $Z$ may be cut out by hypersurfaces of degree $d$ in $\Pbb^n$, and let
\[
s(Z,\Pbb^n)^{\cO(-d)} = a_0 [\Pbb^n]+a_1 [\Pbb^{n-1}] + \cdots + a_n [\Pbb^0]\quad.
\]
Then the numbers $1-a_0$, $d-a_1$, \dots, $d^n-a_n$ form a {\em log-concave
sequence of nonnegative integers with no internal zeros.\/} 
Indeed, if $Z$ may be cut out by hypersurfaces of degree $d$, then the blow-up
of $\Pbb^n$ along $Z$ may be realized as a subvariety of 
$\Pbb^n\times \Pbb(\cO(d)^r)\cong \Pbb^n\times \Pbb^{r-1}$ for some $r$, and
the numbers $N_i$ in Theorem~\ref{thm:enum} may be interpreted as the
multidegrees of the class of this blow-up in $\Pbb^n\times \Pbb^{r-1}$.
These numbers form a log-concave sequence by~\cite[Theorem~21]{MR2904577},
and the statement follows.
\qede\end{remark}

%%%

\section{Segre classes of linear joins}\label{sec:cone}
Next we consider Segre classes of linear joins. The following
situation generalizes slightly the one presented in~\S\ref{sec:intro}; 
this generalization has been useful in applications.  Let $V$ be a
variety, and $Z\subseteq Y=V\times \Pbb^n$ a closed subscheme defined
by a section $s$ of $\cE\otimes \cO(d)$, where $\cE$ is (the pull-back
of) a vector bundle defined on $V$. The situation described
in~\S\ref{sec:intro} corresponds to taking $V$ to be a point. For
any $N\ge n$, we embed $\Pbb^n$ in $\Pbb^N$, for example by
$(x_0:\dots:x_n)\mapsto (x_0:\dots:x_n:0:\dots:0)$; we may then define
a subscheme $Z^{(d)}_N$ of $V\times \Pbb^N$ by using `the same'
section $s$, interpreting the $\cO(d)$ components of $s$ as expressed
in the first
$(n+1)$ homogeneous coordinates of $\Pbb^N$. Geometrically, the scheme
$Z^{(d)}_N$ is supported on the join of $Z\subseteq V\times
\Pbb^n\subseteq V\times \Pbb^N$ and $V\times \Pbb^{N-n-1}$, where
$\Pbb^{N-n-1}$ is spanned by the last $N-n$ homogeneous
coordinates. Thus $Z^{(d)}_n=Z$; but note that the scheme structure on
$Z^{(d)}_N$ along the `vertex' $V\times \Pbb^{N-n-1}$ is not
determined by $Z$ alone---it also depends on the choice of $d$
(cf.~Example~\ref{ex:lembp}).

These linear joins also define a map $\alpha \mapsto \alpha \vee
(V\times \Pbb^m)$, ($m=N-n-1$) from $A_*Z$ to~$A_* Z^{(d)}_N$: if $W$
is a subvariety of $Z$, the join $W\vee (V\times \Pbb^m)$ is a
subvariety of $Z^{(d)}_N$.  Theorem~\ref{thm:cone} is the particular
case corresponding to $V=\{pt\}$ of the following statement.

\begin{theorem}\label{thm:conep}
With notation as above, and letting $H$ denote the hyperplane class:
\begin{equation}\label{eq:coneformp}
s(Z^{(d)}_N,V\times \Pbb^N)^{\cO(-dH)} = \frac{d^{n+1}[V\times \Pbb^m]}{1-dH}
+s(Z,\Pbb^n)^{\cO(-dH)}\vee (V\times \Pbb^m)\quad.
\end{equation}
\end{theorem}

\begin{proof}
We consider the projection $\pp:V\times \Pbb^N \dashrightarrow V\times
\Pbb^n$ with center at $V\times \Pbb^m$; the indeterminacy of $\pp$ is
resolved by blowing up $V\times \Pbb^N$ along $V\times \Pbb^m$.  Let
$\pi: \Til Y \to V\times \Pbb^N$ be this blow-up, $\tilde\pp$ the lift
of $\pp$ to $\Til V$, and let $E$ be the exceptional divisor.
\[
\xymatrix{
E \ar@{^(->}[r] \ar[d]_\rho & \Til V\ar[d]_\pi \ar[drr]^{\tilde\pp} \\
V\times \Pbb^m \ar@{^(->}[r] & V\times \Pbb^N \ar@{-->}[rr]^\pp && V\times \Pbb^n
}
\]
By hypothesis, $Z^{(d)}_N$ is defined by the section $s$ of
$\cE\otimes \cO(d)$ whose zero-scheme is $Z$ in $V\times \Pbb^n$. It
follows that $\pi^{-1}(Z^{(d)}_N) =E\cup \tilde\pp^{-1}(Z)$
set-theoretically; we will refine this statement to a scheme-theoretic
one in a moment.  First, let $\pi'$, resp., $\tilde\pp'$ be the
restrictions of $\pi$, resp., $\tilde\pp$ to $\pi^{-1}(Z^{(d)}_N)$,
and note that
\[
\pi'_* \tilde\pp^{\prime*}([W]) = [W]\vee \Pbb^m
\]
realizes the join operation $A_*Z \to A_* Z^{(d)}_N$.

Let $(x_0:\dots:x_N)$ be homogeneous coordinates in $\Pbb^N$. On open
sets $U$ of a cover of $V$ we may write $s=(F_1,\dots, F_m)$ ($m=\rk
\cE$) with
\[
F_i = F_i(x_0,\dots, x_n)\in \cO_V(U)[x_0,\dots, x_n]
\]
homogeneous polynomials. The center $V\times \Pbb^m$ of the blow-up is
cut out by $x_0,\dots,x_n$, so the ideal of $\pi^{-1}(Z^{(d)}_N)$ in
$\Til V$ is generated over $U$ by
\[
F_i(\eta x_j)=\eta^d F_i(x_j)
\]
in the patch for $\Til V$ obtained by setting one of the $x_j$'s to
$1$ and letting $\eta$ be the coordinate corresponding to the
exceptional divisor. It follows that
\[
\pi^{-1}(Z^{(d)}_N) = dE \cup \tilde\pp^{-1}(Z)
\]
(scheme theoretically). We apply the residual intersection formula
((iv), cf.~\S\ref{sec:ri}) with $D=dE$, $R=\tilde\pp^{-1}(Z)$, and
$\cL=\cO(-dH)$ to obtain
\[
s(\pi^{-1}(Z^{(d)}_N),\Til V)^{\cO(-dH)} = s(dE,\Til V)^{\cO(-dH)} + 
s(\tilde\pp^{-1}(Z),\Til V)^{\cO(d(E-H))}\quad.
\]
By birational invariance ((ii), \S\ref{sec:bi}) and the projection
formula, this yields
\[
s(Z^{(d)}_N,\Pbb^{n+m+1})^{\cO(-dH)} = \pi'_*\left(
s(dE,\Til V)^{\cO(-dH)} + s(\tilde\pp^{-1}(Z),\Til V)^{\cO(d(E-H))}\right)\quad.
\]
We have to understand the push-forward by $\pi'_*$ of the two terms on
the right-hand side.

Concerning the first term, we have $s(dE,\Til V) = \frac{d[E]}{1+dE}$,
and $\pi'$ restricts to $\rho$ on $E$. We have
\[
\rho_* \frac {[E]}{1+E} = s(V\times \Pbb^m,V\times \Pbb^{n+m+1}) = \frac{[V\times \Pbb^m]}{(1+H)^{n+1}}
\]
by the birational invariance of Segre classes and the fact that $V\times \Pbb^m$ is 
regularly embedded in $V\times \Pbb^N$, with normal bundle $\cO(H)^{\oplus (n+1)}$. Hence
\[
\rho_* \frac {d[E]}{1+dE} =  \frac{d^{n+1}[V\times \Pbb^m]}{(1+dH)^{n+1}}\quad.
\]
It follows that
\[
\pi'_* s(dE,\Til V)^{\cO(-dH)} = \left(\frac{d^{n+1}[V\times \Pbb^m]}{(1+dH)^{n+1}}\right)^{\cO(-dH)}
=\frac{d^{n+1}[V\times \Pbb^m]}{1-dH}
\]
as a class in $A_*Z^{(d)}_N$. (Here we have again
used~\cite[Proposition~1]{MR96d:14004}.)  As for the other term, since
$\tilde\pp$ is flat, then (\S\ref{sec:bi})
\[
s(\tilde\pp^{-1}(Z),\Til V)^{\cO(d(H-E))}
=\left(\tilde\pp^{\prime*} s(Z,V\times \Pbb^n)\right)^{\cO(d(H-E))}\quad;
\]
and since $H-E$ is the pull-back of the hyperplane class from $\Pbb^n$
(which we also denote by~$H$), we get
\[
s(\tilde\pp^{-1}(Z),\Til V)^{\cO(d(E-H))}=\tilde\pp^{\prime*} (s(Z,V\times \Pbb^n)^{\cO(-dH)})\quad.
\]
It follows that
\[
\pi'_* s(\tilde\pp^{-1}(Z),\Til V)^{\cO(d(E-H))}=\pi'_*\tilde\pp^{\prime*} (s(Z,V\times \Pbb^n)^{\cO(-dH)})
=s(Z,V\times \Pbb^n)^{\cO(-dH)} \vee (V\times \Pbb^m)
\]
and this concludes the proof.
\end{proof}

For $d=1$, \eqref{eq:coneformp} reproduces Lemma~4.2
in~\cite{MR3383478}, which was stated and used without proof in that
reference. Theorem~\ref{thm:cone} follows from
Theorem~\ref{thm:conep}, by letting $V=$ a point. In this case,
the information of $s$ consists of $m=\rk E$ homogeneous polynomials
$F_1,\dots, F_m\in k[x_0,\dots, x_n]$ of the same degree $d$.

\begin{example}\label{ex:nscon}
Let $Z$ be a nonsingular conic in $\Pbb^2$; then $Z^{(d)}_3$ is
supported on a quadric cone in $\Pbb^3$. By Theorem~\ref{thm:cone},
\[
s(Z^{(d)}_3,\Pbb^3)^{\cO(-dH)} = d^3[\Pbb^0]+s(Z,\Pbb^2)^{\cO(-dH)}\wedge \Pbb^0\quad.
\]
We have
\[
s(Z,\Pbb^2)^{\cO(-dH)} = (1-dH)^{-1} (1+(2-d)H)^{-1} \cap [Z] = [Z]+(2d-2)H\cdot [Z]
=[Z]+(4d-4)[\Pbb^0]
\]
and therefore
\[
s(Z^{(d)}_3,\Pbb^3)^{\cO(-dH)} = 2[\Pbb^2]+(4d-4)[\Pbb^1]+d^3[\Pbb^0]
\]
after push-forward to $\Pbb^3$.
The ordinary Segre class is immediately obtained from this:
\[
s(Z^{(d)}_3,\Pbb^3) = \frac{2[\Pbb^2]}{(1+dH)^2} + \frac{(4d-4)[\Pbb^1]}{(1+dH)^3}
+ \frac{d^3[\Pbb^0]}{(1+dH)^4} 
=2[\Pbb^2]-4[\Pbb^1] +d(d^2-6d+12)[\Pbb^0]
\]
after push-forward to $\Pbb^3$. The case $d=2$ corresponds to the {\em
  reduced\/} quadric cone in $\Pbb^3$.  
\qede\end{example}

In the rest of the section we will focus on the simpler
Theorem~\ref{thm:cone}.  In our view, the most interesting feature of
this statement is that the shape of the expression~\eqref{eq:coneform}
is {\em independent\/} of $N\ge n$; thus it can be taken as an
invariant of the ideal $I=(F_1,\dots, F_m)$. In terms of ordinary
Segre classes, this observation takes the following form.

\begin{theorem}\label{thm:rf}
With notation as above, let $\iota_N: Z^{(d)}_N\hookrightarrow \Pbb^N$
be the inclusion.  Then
\begin{equation}\label{eq:rf}
\iota_{N*} s(Z^{(d)}_N,\Pbb^N) = \frac{A(H)}{(1+dH)^{n+1}} \cap [\Pbb^N]
\end{equation}
where $A(H)$ is a polynomial of degree $n+1$ with nonnegative coefficients,
independent of~$N$.
\end{theorem}

\begin{proof}
We push forward the class to $\Pbb^N$, and write it
`cohomologically'. So
\[
\iota_{N*} s(Z,\Pbb^n)^{\cO(-dH)}=a_0 + a_1 H+\cdots +a_n H^n
\]
is a shorthand for the class $(a_0 + a_1 H+\cdots +a_n H^n)\cap
[\Pbb^n] =a_0 [\Pbb^n] + a_1 [\Pbb^{n-1}] + \cdots + a_n
[\Pbb^0]$. This class is effective (by
Corollary~\ref{cor:constraint}), so the coefficients $a_i$ are all
nonnegative.  By Theorem~\ref{thm:cone},
\[
\iota_{N*} s(Z^{(d)}_N,\Pbb^N)^{\cO(-dH)} = a_0 + a_1 H+\cdots +a_n H^n + 
\frac{d^{n+1}H^{n+1}}{1-dH}\quad.
\]
The ordinary Segre class is obtained by tensoring by $\cO(dH)$:
\[
\iota_{N*} s(Z^{(d)}_N,\Pbb^N) = \frac{a_0}{1+dH} + \cdots + \frac{a_n H^n}{(1+dH)^{n+1}}
+\frac{d^{n+1} H^{n+1}}{(1+dH)^{n+1}}
\]
that is,
\begin{equation}\label{eq:expli}
\iota_{N*} s(Z^{(d)}_N,\Pbb^N) =\frac{a_0 (1+dH)^n + a_1 H(1+dH)^{n-1} + \cdots 
+ a_n H^n + d^{n+1} H^{n+1}}{(1+dH)^{n+1}}
\end{equation}
and this verifies the statement.
\end{proof}

\begin{remark}\label{rem:PIint}
The polynomial $A(H)$ has the following interpretation. Let
$S_Z(H)\in \Zbb[H]$ be the polynomial of degree $\le n$ such that
\[
\iota_* s(Z,\Pbb^n)=S_Z(H) \cap [\Pbb^n]\quad.
\]
Then
\begin{equation}\label{eq:num}
A(H) = \left[ (1+dH)^{n+1} S_Z(H) \right]_n + d^{n+1} H^{n+1}\quad,
\end{equation}
where $[\cdot ]_n$ denotes truncation to $H^n$ of the polynomial
within $[\cdot]$.  (This is obtained from the numerator
of~\eqref{eq:expli} by a computation analogous to the one presented at
the end of~\S\ref{sec:equiv}.) Thus, $A(H)-d^{n+1} H^{n+1}$ is the
unique polynomial of degree $\le n$ such that
\[
(A(H)-d^{n+1} H^{n+1})\cap [\Pbb^n] = (1+dH)^{n+1}\cap \iota_* s(Z,\Pbb^n)\quad.
\]
As observed at the end of~\S\ref{sec:equiv}, this is an effective
class; and indeed $A(H)$ has nonnegative coefficients as proven in
Theorem~\ref{thm:rf}.  Also note that $S_Z(H)$ is determined by $Z$ as
a subscheme of $\Pbb^n$, while $Z^{(d)}_N$ depends on the choice of
degree $d$ for generators of an ideal defining
$Z$. Expressions~\eqref{eq:rf} and~\eqref{eq:num} clarify the
dependence of the Segre class $s(Z^{(d)}_N, \Pbb^N)$ on the scheme $Z$
and the choice of $d$.
\end{remark}

\begin{example}
For the conic in Example~\ref{ex:nscon}, $A(H)=2H+(6d-4)H^2+d^3H^3$.
\qede\end{example}

\begin{remark}
The class $\iota_{N*} s(Z^{(d)},\Pbb^N)$ (for $N\gg 0$) is an
invariant determined by the homogeneous ideal $I=(F_1,\dots, F_m)$
chosen to define $Z$ scheme-theoretically in $\Pbb^n$, subject to the
condition that $\deg F_i = d$ for all $i$.  By Theorem~\ref{thm:rf},
this invariant of $I$ may be interpreted as the result of setting $t=$
hyperplane class $H$ in a well-defined rational function $\zeta_I(t)$
with a single pole at $-1/d$ of order $\le (n+1)$, and numerator of
degree $\le (n+1)$ and with nonnegative coefficients.

Such a `zeta function' $\zeta_I(t)$ can be defined for any homogeneous
ideal $I\subseteq k[x_0,\dots, x_n]$, and we will prove elsewhere that
the essential features verified here for ideals generated by
polynomials of a fixed degree hold in general:
\begin{itemize}
\item $\zeta_I(t)$ is rational;
\item The poles of $\zeta_I(t)$ are at $-1/d_i$, where $d_i$ are
  degrees of polynomials in any generating set for $I$.
\item The numerator of $\zeta_I(t)$ is a polynomial of degree
equal to the degree of the denominator, with nonnegative
coefficients, and with known leading term.
\end{itemize}
In general, not all generators of $I$ will contribute poles to
$\zeta_I(t)$.  It would be very worthwhile providing a complete description 
of the poles of $\zeta_I(t)$ and an effective interpretation of the numerator 
of this function.  At present, such
descriptions are available for the case studied in this note (as
discussed in Remark~\ref{rem:PIint}) and for ideals generated by
monomials, where the information can be obtained from an associated
Newton polytope.  
\qede\end{remark}

%%%

\end{document}